\documentclass[12pt]{amsart}  % draft

\usepackage{custom_preamble} % showkeys

\begin{document}

\title{On Roth's theorem on progressions}

\author{Tom Sanders}
\address{Department of Pure Mathematics and Mathematical Statistics\\
University of Cambridge\\
Wilberforce Road\\
Cambridge CB3 0WB\\
England } \email{t.sanders@dpmms.cam.ac.uk}

\begin{abstract}
We show that if $A \subset \{1,\dots,N\}$ contains no non-trivial three-term arithmetic progressions then $|A|=O(N/\log^{1-o(1)}N)$. 
\end{abstract}

\maketitle

\section{Introduction}

In this paper we prove the following version of Roth's theorem on arithmetic progressions.
\begin{theorem}\label{thm.roth}
Suppose that $A \subset \{1,\dots,N\}$ contains no non-trivial three-term arithmetic progressions.  Then
\begin{equation*}
|A| = O\left(\frac{N (\log \log N)^5 }{\log N}\right).
\end{equation*}
\end{theorem}
There are numerous detailed expositions and proofs of Roth's theorem and the many related results, so we shall not address ourselves to a comprehensive history here.  Briefly, the first non-trivial upper bound on the size of such sets was given by Roth \cite{rot::,rot::0}, and there then followed refinements by Heath-Brown \cite{hea::} and Szemer{\'e}di \cite{sze::2}, and later Bourgain \cite{bou::5}, leading up to the above with the power $1/2$ in place of $1$ (up to doubly logarithmic factors).

Bourgain then introduced a new sampling technique in \cite{bou::1} which was refined in \cite{san::01} to give the previous best bound which had a power of $3/4$ (again up to doubly logarithmic factors) in place of $1$.  The methods of this paper, however, are largely unrelated to these last developments.  We do still use the Bohr set technology of Bourgain \cite{bou::5}, but couple this with two new tools: the first is motivated by the arguments of Katz and Koester in \cite{katkoe::} and is a sort of variant of the Dyson $e$-transform; the second is a result on the $L^p$-invariance of convolutions due to Croot and Sisask \cite{crosis::}.

For comparison with these upper bounds, Salem and Spencer \cite{salspe::} showed that the surface of high-dimensional convex bodies can be embedded in the integers to construct sets of size $N^{1-o(1)}$ containing no three-term progressions, and Behrend \cite{beh::} noticed that spheres are a particularly good choice.  Recently Elkin \cite{elk::} tweaked this further by thickening the spheres to produce the largest known progression-free sets, and his argument was then considerably simplified by Green and Wolf in the very short and readable note \cite{grewol::}.

\section{Notation}

Suppose that $G$ is a finite Abelian group.  We write $M(G)$ for the space of measures on $G$ and given a measure $\mu \in M(G)$ and a function $f \in L^1(\mu)$ we write $fd\mu$ for the measure induced by 
\begin{equation*}
C(G) \rightarrow C(G); g \mapsto \int{g(x)f(x)d\mu(x)}.
\end{equation*}
Given a non-empty set $A \subset G$ we write $\mu_A$ for the uniform probability measure supported on $A$, that is the measure assigning mass $|A|^{-1}$ to each $x \in A$, so that $\mu_G$ is Haar probability measure on $G$.  The significance of this measure is that it is the unique (up to scaling) translation invariant measure on $G$: for $x \in G$ and $\mu \in M(G)$ we define $\tau_x(\mu)$ to be the measure induced by
\begin{equation*}
C(G) \rightarrow C(G); g \mapsto \int{g(y)d\mu(y+x)},
\end{equation*}
and it is easy to see that $\tau_x(\mu_G)=\mu_G$ for all $x \in G$. 

Translation can be usefully averaged by convolution: given two measures $\mu,\nu \in M(G)$ define their convolution $\mu \ast \nu$ to be the measure induced by
\begin{equation*}
C(G) \rightarrow C(G); g \mapsto \int{g(x+y)d\mu(x)d\nu(y)}.
\end{equation*}

We use Haar measure to pass between the notion of function $f\in L^1(\mu_G)$ and measure $\mu \in M(G)$.  Indeed, since $G$ is finite we shall often identify $\mu$ with $d\mu/d\mu_G$, the Radon-Nikodym derivate of $\mu$ with respect to $\mu_G$.  In light of this we can easily extend the notion of translation and convolution to $L^1(\mu_G)$: given $f \in L^1(\mu_G)$ and $x \in G$ we define the translation of $f$ by $x$ point-wise by
\begin{equation*}
\tau_x(f)(y):=\frac{d(\tau_x(fd\mu_G))}{d\mu_G}(y) = f(x+y) \textrm{ for all } y \in G;
\end{equation*}
and given $f,g \in L^1(\mu_G)$ we define convolution of $f$ and $g$ point-wise by
\begin{equation*}
f \ast g(x):=\frac{d((fd\mu_G) \ast (g d\mu_G))}{d\mu_G}(x) = \int{f(y)g(x-y)d\mu_G(y)},
\end{equation*}
and similarly for the convolution of $f \in L^1(\mu_G)$ with $\mu \in M(G)$. 

Convolution operators can be written in a particularly simple form with respect to the Fourier basis which we now recall.  We write $\wh{G}$ for the dual group, that is the finite Abelian group of homomorphisms $\gamma:G \rightarrow S^1$, where $S^1:=\{z \in \C:|z|=1\}$.  Given $\mu \in M(G)$ we define $\wh{\mu} \in \ell^\infty(\wh{G})$ by
\begin{equation*}
\wh{\mu}(\gamma):=\int{\overline{\gamma}d\mu} \textrm{ for all } \gamma \in \wh{G},
\end{equation*}
and extend this to $f \in L^1(G)$ by $\wh{f}:=\wh{fd\mu_G}$. It is easy to check that $\wh{\mu \ast \nu} = \wh{\mu}\cdot \wh{\nu}$ for all $\mu,\nu \in M(G)$ and $\wh{f \ast g} = \wh{f} \cdot \wh{g}$ for all $f,g \in L^1(\mu_G)$.

Throughout the paper $C$s will denote absolute, effective, but unspecified constants of size greater than $1$ and $c$s will denote the same of size at most $1$.  Typically the constants will be subscripted according to the result from which they come and superscripted within arguments.

\section{Fourier analysis on Bohr sets}

Fourier analysis on Bohr sets was introduced to additive combinatorics by Bourgain in \cite{bou::5}, and has since become a fundamental tool.  The material is standard so we shall import the results we require from \cite{san::01} without comment; for a more detailed discussion the reader may wish to consult the book \cite{taovu::} of Tao and Vu.

A set $B$ is called a \emph{Bohr set} with \emph{frequency set} $\Gamma \subset \wh{G}$ and \emph{width function} $\delta \in (0,2]^\Gamma$ if
\begin{equation*}
B=\{x \in G: |1-\gamma(x)| \leq \delta_\gamma \textrm{ for
all }\gamma \in \Gamma\}.
\end{equation*}
The size of the set $\Gamma$ is called the \emph{rank} of $B$ and is denoted $\rk(B)$.

There is a natural way of dilating Bohr sets which will be of particular use to us.  Given such a $B$, and $\rho \in \R^+$ we shall write $B_\rho$ for the Bohr set with frequency set $\Gamma$ and width function\footnote{Technically width function $\gamma \mapsto \min\{\rho \delta_\gamma,2\}$.} $\rho\delta$ so that, in particular, $B=B_1$.   

With these dilates we say that a Bohr set $B'$ is a \emph{sub-Bohr set} of another Bohr set $B$, and write $B' \leq B$, if 
\begin{equation*}
B'_\rho \subset B_{\rho} \textrm{ for all } \rho \in \R^+.
\end{equation*}
Finally, we write $\beta_\rho$ for the probability measure induced on $B_\rho$ by $\mu_G$, and $\beta$ for $\beta_1$. 

\subsection{Size and regularity of Bohr sets} The rank of a Bohr set is closely related to its dimension: a Bohr set $B$ is said to be \emph{$d$-dimensional} if
\begin{equation*}
\mu_G(B_{2\rho}) \leq 2^d\mu_G(B_\rho) \textrm{ for all } \rho \in (0,1],
\end{equation*}
and we have the following standard averaging argument, see \cite[Lemma 4.20]{taovu::}.
\begin{lemma}[Dimension of Bohr sets]\label{lem.bohrsize} Suppose that $B$ is a rank $k$ Bohr set.  Then it is $O(k)$-dimensional.
\end{lemma}

A key observation of \cite{bou::5} was that some Bohr sets behave better than others: a $d$-dimensional Bohr set is said to be \emph{$C$-regular} if
\begin{equation*}
\frac{1}{1+Cd |\eta|} \leq \frac{\mu_G(B_{1+\eta})}{\mu_G(B_1)} \leq 1+Cd|\eta| \textrm{ for all } \eta \textrm{ with } |\eta| \leq 1/Cd.
\end{equation*}
Crucially, regular Bohr sets are plentiful:
\begin{lemma}[Regular Bohr sets]\label{lem.ubreg} There is an absolute constant $C_\mathcal{R}$ such that whenever $B$ is a Bohr set, there is some $\lambda \in [1/2,1)$ such that $B_\lambda$ is $C_\mathcal{R}$-regular.
\end{lemma}
The result is proved by a covering argument due to Bourgain \cite{bou::5}; for details one may also consult \cite[Lemma 4.25]{taovu::}.  For the remainder of the paper we shall say \emph{regular} for $C_\mathcal{R}$-regular.

\subsection{The large spectrum} Given a probability measure $\mu$, a function $f \in L^1(\mu)$ and a parameter $\epsilon \in (0,1]$ we define the \emph{$\epsilon$-spectrum of $f$ w.r.t. $\mu$} to be the set
\begin{equation*}
\Spec_\epsilon(f,\mu):=\{\gamma \in \wh{G}: |(fd\mu)^\wedge(\gamma)| \geq \epsilon\|f\|_{L^1(\mu)}\}.
\end{equation*}
This definition extends the usual one from the case $\mu=\mu_G$.  We shall need a local version of a result of Chang \cite{cha::0} for estimating the `complexity' or `entropy' of the large spectrum.  

Conceptually the next definition is inspired by the discussion of quadratic rank Gowers and Wolf give in \cite{gowwol::}.  The \emph{$(K,\mu)$-relative entropy} of a set $\Gamma$ is the size of the largest subset $\Lambda\subset \Gamma$ such that
\begin{equation*}
\int{\prod_{\lambda \in \Lambda}{(1+\Re \omega(\lambda)\lambda)}d\mu} \leq \exp(K)  \textrm{ for all } \omega:\Lambda \rightarrow D,
\end{equation*}
where $D:=\{z \in \C:|z|\leq 1\}$.  The definition is essentially relativising the notion of being dissociated (if $\mu=\mu_G$ and $K=0$ it is precisely this), but the reader does not need to have a deep understanding for the purposes of this paper as it is only used to couple the next two results from \cite{san::01} into Lemma \ref{lem.qw}.  
\begin{lemma}[The Chang bound, {\cite[Lemma 4.6]{san::01}}]\label{lem.changbd}  Suppose that $0 \not \equiv f \in L^2(\mu)$.  Then $\Spec_\epsilon(f,\mu)$ has $(1,\mu)$-relative entropy $O( \epsilon^{-2}\log 2\|f\|_{L^2(\mu)}\|f\|_{L^1(\mu)}^{-1})$.
\end{lemma}
Low entropy sets of characters are majorised by large Bohr sets, a fact encoded in the following lemma.
\begin{lemma}[{\cite[Corollary 6.4]{san::01}}]\label{lem.dis}
Suppose that $B$ is a regular $d$-dimensional Bohr set and $\Delta$ is a set of characters with $(\eta,\beta)$-relative entropy $k$.  Then there is a Bohr set $B' \leq B$ with
\begin{equation*}
\rk(B') \leq \rk(B)+k \textrm{ and } \mu_{B}(B') \geq (\eta/2dk)^{O(d)}(1/2k)^{O(k)}
\end{equation*}
such that $|1-\gamma(x)| \leq 1/2$ for all $x \in B'$ and $\gamma \in \Delta$.
\end{lemma}

\subsection{The energy increment method}  The final lemma of the section encodes the Heath-Brown-Szemer{\'e}di energy increment technique from \cite{hea::,sze::2} which shows how to get a density increment on a Bohr set from large energy on a large spectrum.
\begin{lemma}\label{lem.qw}
Suppose that $B$ is a regular $d$-dimensional Bohr set, $A \subset B$ has density $\alpha>0$, $B'\subset B_{\rho'}$ is a regular rank $k$ Bohr set, $T \subset B'$ has relative density $\tau$ and
\begin{equation*}
\sum_{\gamma \in\Spec_\eta(1_T,\beta')}{|((1_A-\alpha)1_B)^\wedge(\gamma)|^2} \geq \nu \alpha^2\mu_G(B).
\end{equation*}
Then there is a regular Bohr set $B''$ with
\begin{equation*}
\rk(B'') \leq k+ O(\eta^{-2}\log 2\tau^{-1}) \textrm{ and } \mu_{B'}(B'') \geq \left(\frac{\eta}{2k\log 2\tau^{-1}}\right)^{O(k+\eta^{-2}\log2\tau^{-1})}
\end{equation*}
such that $\|1_A \ast \beta''\|_{L^\infty(\mu_G)} \geq \alpha (1+\Omega(\nu))$ provided ${\rho'} \leq c_{\ref{lem.qw}}\nu\alpha/d$.
\end{lemma}
\begin{proof} By Lemma \ref{lem.changbd} the set $\Spec_\eta(1_T,\beta')$ has $(1,\beta')$-relative entropy $O(\eta^{-2}\log 2\tau^{-1})$.  The dimension of $B'$ is $O(k)$ so it follows by Lemma \ref{lem.dis} that there is a Bohr set $B'' \leq B'$ with
\begin{equation*}
\rk(B'') \leq k +O(\eta^{-2}\log 2\tau^{-1}) \textrm{ and } \mu_{B'}(B'') \geq \left(\frac{\eta}{2k\log 2\tau^{-1}}\right)^{O(k+\eta^{-2}\log2\tau^{-1})}
\end{equation*}
such that
\begin{equation*}
\Spec_\eta(1_T,\beta') \subset \{\gamma: |1-\gamma(x)| \leq 1/2 \textrm{ for all } x \in B''\}.
\end{equation*}
By the triangle inequality and Parseval's theorem it follows that
\begin{eqnarray*}
\Omega(\nu \alpha^2\mu_G(B)) & = &\sum_{\gamma \in\wh{G}}{|((1_A-\alpha)1_B)^\wedge(\gamma)|^2|\wh{\beta''}(\gamma)|^2}\\ & = & \|(1_A - \alpha1_B)\ast \beta''\|_{L^2(\mu_G)}^2.
\end{eqnarray*}
Since $B'' \leq B' \subset B_{\rho'}$ and $B$ is regular we have that
\begin{equation*}
\|(1_A - \alpha1_B)\ast \beta''\|_{L^2(\mu_G)}^2 = \|1_A \ast \beta''\|_{L^2(\mu_G)}^2 -\alpha^2\mu_G(B) + O(\alpha {\rho'} d\mu_G(B)).
\end{equation*}
It follows that if ${\rho'}$ is sufficiently small then
\begin{equation*}
 \|1_A \ast \beta''\|_{L^2(\mu_G)}^2\geq \alpha^2(1+\Omega(\nu))\mu_G(B)
\end{equation*}
and we get the result by H{\"o}lder's inequality.
\end{proof}

\section{Katz-Koester and the Dyson $e$-transform}\label{sec.kk}

In \cite{katkoe::} Katz and Koester introduced a new way of transforming sumsets.  This method has seen impressive applications in, for example, \cite{sch::1} and \cite{schshk::}, and is particularly ripe for iteration.  The arguments of this section evolved from these Katz-Koester techniques but in their final form may be seen to have more in common with the Dyson $e$-transform (see \emph{e.g.} \cite[\S 5.1]{taovu::}).  In any case, from our perspective what is important is that it provides a sort of density increment without the cost of passing to an approximate subgroup.

Specifically our aim is to transform the set $A$ in Roth's theorem into two sets $L$ and $S$ where $L$ is thick, $S$ is not too thin, and $L+S \subset A-2.A$.  This dovetails with the regime of strength of the results in the next section.

The main idea is to construct such sets $L$ and $S$ iteratively using the Katz-Koester transformation.  Suppose that $L,S,A$ and $A'$ are sets of density $\lambda$, $\sigma$, $\alpha$ and $\alpha'$ respectively and $L+S \subset A+A'$.  Unless $A$ is `quite structured' one expects there to be very few $x$ for which
\begin{equation*}
1_{L} \ast 1_{-A}(x) \geq \alpha/2;
\end{equation*}
on the other hand, by averaging, there are many $x \in G$ such that
\begin{equation*}
1_{-S}\ast 1_{A'}(x) \geq \sigma\alpha'/2.
\end{equation*}
It follows that unless $A$ is `quite structured' one may find an $x \in G$ such that
\begin{equation*}
1_L \ast 1_{-A}(x) \leq \alpha/2 \textrm{ and } 1_{-S}\ast 1_{A'}(x) \geq \sigma\alpha'/2.
\end{equation*}
Now, if we put 
\begin{equation*}
L':= L \cup (x+A) \textrm{ and } S':=S \cap (A'-x),
\end{equation*}
then we have
\begin{equation*}
\mu_G(L') \geq \mu_G(L) + \mu_G(x+A) - 1_L \ast 1_{-A}(x) \geq \lambda + \alpha/2 \textrm{ and } \mu_G(S') \geq \alpha'\sigma/2,
\end{equation*}
and also
\begin{equation*}
L'+S' \subset (L+S') \cup ((x+A) + S')\subset (L+S) \cup (x+A+A'-x) \subset A+A'.
\end{equation*}
We see that unless $A$ is quite structured we have a new pair $(L',S')$ whose sumset is contained in $A+A'$, but for which $L'$ is somewhat larger (than $L$) while $S'$ is not too much smaller (than $S$).

The actual result we require is the following relativised and weighted version of the above.
\begin{proposition}\label{prop.inneritapplied}
Suppose that $B$ is a regular $d$-dimensional Bohr set, $B'$ is a regular rank $k$ Bohr set with $B' \subset B_{\rho'}$,$B''\subset B_{\rho''}'$, $A \subset B$ has relative density $\alpha$ and $A' \subset B'$ has relative density $\alpha'$. Then either
\begin{enumerate}
\item there is a regular Bohr set $B'''$ of rank at most $k+O(\alpha^{-1}\log2\alpha'^{-1})$ with
\begin{equation*}
\mu_{B'}(B''') \geq \left(\frac{\alpha}{2k\log 2\alpha'^{-1}}\right)^{O(k+\alpha^{-1}\log2\alpha'^{-1})}
\end{equation*}
and $\|1_A \ast \beta'''\|_{L^\infty(\mu_G)} \geq \alpha(1+\Omega(1))$;
\item or there are sets $L\subset B$ and $S\subset B''$ with $\beta(L) =\Omega(1)$ and $\beta''(S) \geq (\alpha'/2)^{O(\alpha^{-1})}$ such that
\begin{equation*}
1_L \ast (1_Sd\beta'')(x) \leq C_{\ref{prop.inneritapplied}}\alpha^{-1}\mu_{B'}(B'')^{-1} 1_A \ast (1_{A'}d\beta')(x)
\end{equation*}
for all $x \in G$;
\end{enumerate}
provided $\rho' \leq c_{\ref{lem.innerit}}\alpha/d$ and $\rho'' \leq c_{\ref{lem.innerit}}\alpha'/k$.
\end{proposition}
The proof is an iteration of the following lemma.
\begin{lemma}\label{lem.innerit}
Suppose that $B$ is a regular $d$-dimensional Bohr set, $B'$ is a regular rank $k$ Bohr set with $B' \subset B_{\rho'}$,$B''\subset B_{\rho''}'$, $A \subset B$ has relative density $\alpha$ and $A' \subset B'$ has relative density $\alpha'$.

If, additionally, there is a set $L \subset B$ of relative density $\lambda$ and $S \subset B''$ of relative density $\sigma$, then either
\begin{enumerate}
\item there is a regular Bohr set $B'''$ of rank at most $k+O(\alpha^{-1}\log2\alpha'^{-1})$ with
\begin{equation*}
\mu_{B'}(B''') \geq \left(\frac{\alpha}{2k\log 2\alpha'^{-1}}\right)^{O(k+\alpha^{-1}\log2\alpha'^{-1})}
\end{equation*}
and $\|1_A \ast \beta'''\|_{L^\infty(\mu_G)} \geq \alpha(1+\Omega(1))$;
\item or there are sets $L'\subset B$ and $S'\subset B''$ with $\beta(L') \geq \lambda + \alpha/4$ and $\beta''(S') \geq \alpha'\sigma/2$ such that
\begin{equation*}
1_{L'} \ast (1_{S'}d\beta'')(x) \leq 1_L \ast (1_Sd\beta'')(x) +  \mu_{B'}(B'')^{-1}1_{A} \ast (1_{A'}d\beta')(x)
\end{equation*}
for all $x \in G$;
\end{enumerate}
provided $\lambda \leq c_{\ref{lem.innerit}}$, $\rho' \leq c_{\ref{lem.innerit}}\alpha/d$ and $\rho'' \leq c_{\ref{lem.innerit}}\alpha'/k$.
\end{lemma}
\begin{proof}
We put
\begin{equation*}
\mathcal{L}:=\{x \in B': 1_{-L} \ast (1_{A}d\beta)(-x) \geq \alpha/2\},
\end{equation*}
and split into two cases.  First, when $\beta'(\mathcal{L})$ is large we shall show that $A$ has a density increment on a Bohr set; secondly, when it is small we shall proceed as per the heuristic at the start of the section.
\begin{case*}$\beta'(\mathcal{L})\geq \alpha'/8$
\end{case*}
\begin{proof}  This is a textbook translation of a physical space condition into a density increment via the Fourier transform.  We consider the inner product
\begin{equation*}
\alpha\beta'(\mathcal{L})/2 \leq \langle 1_{-L} \ast (1_A d\beta),1_{-\mathcal{L}}\rangle_{L^2(\beta')}.
\end{equation*}
By regularity we have that if $\rho'$ is sufficiently small then 
\begin{equation*}
|\langle 1_{-L} \ast \beta,1_{-\mathcal{L}}\rangle_{L^2(\beta')} -\lambda \beta'(\mathcal{L})| \leq \beta'(\mathcal{L})/4.
\end{equation*}
It follows by the triangle inequality that
\begin{equation*}
| \langle 1_{-L} \ast( (1_A -\alpha)d\beta),1_{-\mathcal{L}}\rangle_{L^2(\beta')}| \geq \alpha\beta'(\mathcal{L})(1/4  -\lambda) \geq \alpha \beta'(\mathcal{L}')/8 
\end{equation*}
if $\lambda$ is sufficiently small.  By Fourier inversion and rescaling we then have
\begin{equation*}
\left|\sum_{\gamma \in \wh{G}}{\wh{1_{-L}}(\gamma)(1_A -\alpha1_B)^\wedge(\gamma)\overline{\wh{1_{-\mathcal{L}}d\beta'}(\gamma)}}\right| \geq \alpha \beta'(\mathcal{L})\mu_G(B)/8.
\end{equation*}
By Cauchy-Schwarz and Parseval on the sum of $|\wh{1_{-L}}(\gamma)|^2$ we get that
\begin{equation*}
\sum_{\gamma \in \wh{G}}{|(1_A -\alpha1_B)^\wedge(\gamma)|^2|\wh{1_{-\mathcal{L}}d\beta'}(\gamma)|^2} \geq \alpha^2\beta'(\mathcal{L})^2\mu_G(B)/64.
\end{equation*}
On the other hand, Parseval's theorem tells us that if $\eta:=\sqrt{\alpha}/16$ then
\begin{equation*}
\sum_{\gamma \not\in \Spec_\eta(1_{-\mathcal{L}},\beta')}{|(1_A -\alpha1_B)^\wedge(\gamma)|^2|\wh{1_{-\mathcal{L}}d\beta'}(\gamma)|^2} \leq \alpha^2 \beta'(\mathcal{L})^2\mu_G(B)/16^2.
\end{equation*}
Thus, by the triangle inequality, and since $|\wh{1_{-\mathcal{L}}d\beta'}(\gamma)| \leq \beta'(\mathcal{L})$ we get that
\begin{equation*}
\sum_{\gamma \in \Spec_\eta(1_{-\mathcal{L}},\beta')}{|((1_A -\alpha)1_B)^\wedge(\gamma)|^2} = \Omega(\alpha^2\mu_G(B)).
\end{equation*}
It follows by Lemma \ref{lem.qw} that we are in the first case of the lemma provided $\rho'$ is sufficiently small.
\end{proof}
\begin{case*}$\beta'(\mathcal{L}) \leq \alpha'/8$
\end{case*}
\begin{proof}
First we show that the set
\begin{equation*}
\mathcal{S}:=\{x \in B':(1_{-S}d\beta'') \ast 1_{A'}(x)\geq \alpha' \sigma/2\}
\end{equation*}
is large by averaging.  In particular, 
\begin{equation*}
\beta'(\mathcal{S})\sigma + \alpha'\sigma/2 \geq \int{(1_{-S}d\beta'') \ast 1_{A'}d\beta'} = \int{1_{A'} d((1_{-S}d\beta'')\ast \beta')}.
\end{equation*}
Of course, by regularity we have that
\begin{equation*}
\|(1_{-S}d\beta'') \ast \beta' - \sigma \beta'\| = O(\sigma k \rho''),
\end{equation*}
whence
\begin{equation*}
\beta'(\mathcal{S}) \geq \alpha'/2 - O(k\rho'') \geq \alpha'/4
\end{equation*}
provided $\rho''$ is sufficiently small.  Now, since $\mathcal{L}$ is assumed to be so small there must be some $x \in \mathcal{S} \setminus \mathcal{L}$; we put
\begin{equation*}
L':=L \cup ((x+A)\cap B) \textrm{ and } S':=S \cap (A'-x).
\end{equation*}
Now, $L' \subset B$ and since $x \not \in \mathcal{L}$ we have
\begin{equation*}
\beta(L \cap (x+A)) = \beta((-L) \cap (-x-A)) = 1_{-L} \ast (1_Ad\beta)(-x) \leq \alpha/2.
\end{equation*}
But then
\begin{equation*}
\beta(L') \geq \lambda + \beta((x+A) \cap B) - \alpha/2 \geq \lambda + \alpha/2 - O(d\rho') \geq \lambda +\alpha/4
\end{equation*}
provided $\rho'$ is sufficiently small.  Additionally $S' \subset S \subset B''$ and
\begin{equation*}
\beta''(S') = \beta''(S \cap (A'-x)) =  (1_{-S}d\beta'') \ast 1_{A'} (x) \geq \alpha'\sigma/2,
\end{equation*}
and finally
\begin{eqnarray*}
1_{L'} \ast (1_{S'}d\beta'') &\leq & 1_L \ast (1_{S'}d\beta'') + 1_{x+A} \ast (1_{S'}d\beta'')\\ & \leq & 1_L \ast (1_Sd\beta'') + \mu_{G}(B'')^{-1}1_{x+A} \ast 1_{A'-x}\\ & = & 1_L \ast (1_Sd\beta'') + \mu_{B'}(B'')^{-1}1_{A} \ast (1_{A'}d\beta')
\end{eqnarray*}
as required.
\end{proof}
\end{proof}
\begin{proof}[Proof of Proposition \ref{prop.inneritapplied}]
We produce a sequence of sets $(L_i)_i$ and $(S_i)_i$ iteratively with $L_i \subset B$, $S_i \subset B''$, $\lambda_i:=\beta(L_i)$ and $\sigma_i:=\beta''(S_i)$ such that
\begin{equation}\label{eqn.u}
\lambda_i \geq \alpha i /4 \textrm{ and } \sigma_i \geq (\alpha'/2)^{i+1}
\end{equation}
and
\begin{equation}\label{eqn.v}
1_{L_i} \ast (1_{S_i}d\beta'') \leq i\mu_{B'}(B'')^{-1}1_A \ast (1_{A'}d\beta').
\end{equation}
To initialise the iteration we consider the inner product
\begin{equation*}
\int{1_{A'} \ast \beta''d\beta'} = \alpha' + O(\rho '' k).
\end{equation*}
Thus, if $\rho''$ is sufficiently small then it follows that there is some $x \in B' \subset B$ such that
\begin{equation*}
\beta''(B''\cap (A'-x)) \geq \alpha'/2;
\end{equation*}
We put $L_0:=\emptyset$ and $S_0:=B'' \cap (A'-x)$ and note that this satisfies (\ref{eqn.u}) and (\ref{eqn.v}).

We now repeatedly apply Lemma \ref{lem.innerit}.  If at any point we are in the first case of that lemma then we terminate in the first case here; otherwise we have the sequence as required.  This process terminates after some $i_0=O(\alpha^{-1})$ steps, when $\lambda_i > c_{\ref{lem.innerit}} = \Omega(1)$.  We set $L:=L_{i_0}$ and $S:=S_{i_0}$ and the result is proved.
\end{proof}

\section{A consequence of the Croot-Sisask lemma}\label{sec.bcs}

In light of the previous section, rather than counting three-term progressions by examining the inner product $\langle 1_A \ast 1_A,1_{2.A}\rangle_{L^2(\mu_G)}$, we shall be able to examine (a relativised version of) $\langle 1_L \ast 1_S,1_{2.A}\rangle_{L^2(\mu_G)}$ where $L$ has density $\Omega(1)$, and $S$, of density $\sigma$, is potentially thin but not too thin.  To do this we shall find a Bohr set $B$ such that
\begin{equation}\label{eqn.uq}
 \|1_L \ast \mu_S \ast \beta - 1_L \ast \mu_S\|_{L^p(\mu_G)} \leq \epsilon,
\end{equation}
so that
\begin{equation*}
|\langle 1_L \ast 1_S \ast \beta,1_{2.A} \rangle_{L^2(\mu_G)} - \langle 1_L\ast 1_S,1_{2.A}\rangle_{L^2(\mu_G)}| \leq \epsilon \sigma \|1_{2.A}\|_{L^{p/(p-1)}(\mu_G)}.
\end{equation*}
If the error is small enough this will give rise to a density increment on $B$; to get a sense of how small it needs to be we think of the second term on the left as being typically of size $\mu_G(L)\sigma\alpha=\Omega(\sigma\alpha)$.  Now,
\begin{enumerate}
\item \label{pt.1} if $p=2$ then $\|1_{2.A}\|_{L^{p/(p-1)}(\mu_G)}= \alpha^{1/2}$ and we would need $\epsilon \sim \alpha^{1/2}$ for the error term not to swamp the main term;
\item \label{pt.2} if $p \sim \log \alpha^{-1}$ then $\|1_{2.A}\|_{L^{p/(p-1)}(\mu_G)} \sim \alpha$ so we would only need $\epsilon \sim 1$ for the error term not to swamp the main term.
\end{enumerate}
Of course, which of these two ranges to use depends on how the size of the Bohr set found varies with $p$ and $\epsilon$.  We shall use an argument of Croot and Sisask \cite{crosis::} to show that we can take
\begin{equation}\label{eqn.bq}
\mu_G(B) \geq \exp(-O(\epsilon^{-2}p\log \sigma^{-1}))
\end{equation}
in (\ref{eqn.uq}), and so in particular case (\ref{pt.2}) above leads to a much larger Bohr set.

This argument of Croot and Sisask is an important new approach for studying the $L^p$-invariance of convolutions.  It relies on random sampling in physical space to approximate a convolution by a small number of translates and works for general groups, not just Abelian ones.

We shall now record a version of their result which will be particularly useful to us.  For completeness -- and since it is simple -- we include the proof of the result as well.
\begin{lemma}[Croot-Sisask]\label{lem.cs}  Suppose that $G$ is a finite Abelian group, $f \in L^p(\mu_G)$ and $A,S \subset G$ have $\mu_G(S+A) \leq K\mu_G(A)$.  Then there is an $s \in S$ and a set $T \subset S$ with $\mu_S(T)  \geq (2K)^{-O(\epsilon^{-2}p)}$ such that
\begin{equation*}
\|\tau_t(f \ast \mu_A) - f \ast \mu_A\|_{L^p(\mu_G)} \leq \epsilon \|f\|_{L^p(\mu_G)} \textrm{ for all } t \in T-s.
\end{equation*}
\end{lemma}
\begin{proof}
Let $z_1,\dots,z_k$ be independent uniformly distributed $A$-valued random variables, and for each $y \in G$ define $Z_i(y):=\tau_{-z_i}(f)(y) - f \ast \mu_A(y)$.  For fixed $y$, the variables $Z_i(y)$ are independent and have mean zero, so it follows by the Marcinkiewicz-Zygmund inequality, with constants due to Yao-Feng and Han-Ying \cite[Theorem 2]{yaohan::}, that
\begin{equation*}
\| \sum_{i=1}^k{Z_i(y)}\|_{L^p(\mu_A^k)}^p \leq  O(p)^{p/2}k^{p/2-1}\sum_{i=1}^k{\int{|Z_i(y)|^p}d\mu_A^k}.
\end{equation*}
Integrating over $y$ and interchanging the order of summation we get
\begin{equation}\label{eqn.khin}
\int{\| \sum_{i=1}^k{Z_i(y)}\|_{L^p(\mu_A^k)}^pd\mu_G(y)} \leq O(p)^{p/2}k^{p/2-1}\int{\sum_{i=1}^k{\int{|Z_i(y)|^p}d\mu_G(y)}d\mu_{A}^k}.
\end{equation}
On the other hand,
\begin{equation*}
\left(\int{|Z_i(y)|^pd\mu_G(y)}\right)^{1/p} =\|Z_i\|_{L^p(\mu_G)} \leq \|\tau_{-z_i}(f)\|_{L^p(\mu_G)}^p + \|f \ast \mu_A\|_{L^p(\mu_G)} \leq 2\|f\|_{L^p(\mu_G)}
\end{equation*}
by the triangle inequality.  Dividing (\ref{eqn.khin}) by $k^p$ and inserting the above and the expression for the $Z_i$s we get that
\begin{equation*}
\int{\int{\left|\frac{1}{k}\sum_{i=1}^k{\tau_{-z_i}(f)(y)} - f \ast \mu_A(y)\right|^pd\mu_G(y)}d\mu_A^k(z)}=O(pk^{-1}\|f\|_{L^p(\mu_G)}^2)^{p/2}.
\end{equation*}
Pick $k=O(\epsilon^{-2}p)$ such that the right hand side is at most $(\epsilon \|f\|_{\ell^p(G)}/4)^p$ and write $L$ for the set of $x=(x_1,\dots,x_k) \in A^k$ for which the integrand above is at most $(\epsilon \|f\|_{\ell^p(G)}/2)^p$; by averaging $\mu_A^k(L^c) \leq 2^{-p}$ and so $\mu_A^k(L) \geq 1-2^{-p} \geq 1/2$.

Now, $\Delta:=\{(s,\dots,s): s \in S\}$ has $L+\Delta  \subset (A+S)^k$, whence $\mu_{G^k}(L + \Delta) \leq 2K^k\mu_{G^k}(L)$ and so
\begin{equation*}
\langle \mu_\Delta \ast \mu_{-\Delta},1_{-L} \ast 1_{L}\rangle_{L^2(\mu_{G^k})} = \|1_L \ast \mu_\Delta \|_{L^2(\mu_{G^k})}^2\geq \mu_{G^k}(L)/2K^k,
\end{equation*}
by the Cauchy-Schwarz inequality since the adjoint of $g \mapsto 1_{L} \ast g$ is $g \mapsto 1_{-L} \ast g$ and similarly for $g \mapsto g \ast \mu_\Delta$.

By averaging it follows that at least $1/2K^k$ of the pairs $(z,y) \in \Delta^2$ have $1_{-L} \ast 1_{L}(z-y)>0$, and hence there is some $s \in S$ such that there is a set $T\subset S$ with $\mu_S(T) \geq 1/2K^k$ and $1_{-L} \ast 1_{L}(t,\dots,t)>0$ for all $t \in T-s$. 

Thus for each $t \in T-s$ there is some $z(t) \in L$ and $y(t) \in L$ such that $y(t)_i=z(t)_i+t$ for all $i$.  But then by the triangle inequality we get that
\begin{eqnarray*}
\|\tau_{-t}(f \ast \mu_A) - f \ast \mu_A\|_{L^p(\mu_G)}& \leq &\|\tau_{-t}\left(\frac{1}{k}\sum_{i=1}^k{\tau_{-z(t)_i}(f)}\right) -  f \ast \mu_A\|_{L^p(\mu_G)}\\&&+\|\tau_{-t}\left(\frac{1}{k}\sum_{i=1}^k{\tau_{-z(t)_i}(f)} - f \ast \mu_A \right)\|_{L^p(\mu_G)}.
\end{eqnarray*}
However, since $\tau_t$ is isometric on $L^p(\mu_G)$ we see that
\begin{eqnarray*}
\|\tau_t(f \ast \mu_A) - f \ast \mu_A\|_{L^p(\mu_G)} &\leq & \|\frac{1}{k}\sum_{i=1}^k{\tau_{-y(t)_i}(f)}-  f \ast \mu_A\|_{L^p(\mu_G)}\\&&+\|\frac{1}{k}\sum_{i=1}^k{\tau_{-z(t)_i}(f)} - f \ast \mu_A\|_{L^p(\mu_G)},
\end{eqnarray*}
and we are done since $z(t),y(t) \in L$.
\end{proof}
The quantitatively weaker arguments of \cite{bou::4}, and the usual Bogolyubov-Chang argument in the case $p=2$ actually endow $T$ with the structure of a Bohr set, while the set we found has, a priori, no structure.  Croot and Sisask noted that this could, to some degree, be recovered by taking repeated sumsets, and we shall couple this idea with Chang's theorem to get the necessary strength in our corollary.

This may sound like we can't have gained anything over the usual multi-sum version of the Bogolyubov-Chang argument.  However, we do get some extra strength  from the fact that we are in some sense able to increase the number of summands without decreasing the (higher order) additive energy or having the individual summands become too thin.  A similar sort of observation is exploited by Schoen in \cite{sch::1} (see also \cite{cwasch::}) for the purpose of proving a remarkable Fre{\u\i}man-type theorem.
\begin{corollary}\label{cor.cs} Suppose that $B$ is a regular $d$-dimensional Bohr set, $B' \subset B_{{\rho'}}$ is a regular rank $k$ Bohr set, $L,A \subset B$ have relative densities $\lambda$ and $\alpha$ respectively, $S \subset B'$ has relative density $\sigma$.  Then either
\begin{enumerate}
\item (large inner product)
\begin{equation*}
\langle 1_L \ast (1_Sd\beta'),1_{A} \rangle_{L^2(\beta)} \geq \lambda \sigma \alpha/2;
\end{equation*}
\item (density increment) or there is a regular Bohr set $B'''$ and an 
\begin{equation*}
m=O(\lambda^{-2}(\log2\lambda^{-1}\alpha^{-1})^2(\log 2\alpha^{-1})(\log 2\sigma^{-1}))
\end{equation*}
with $\rk(B''') \leq k + m$ and $\mu_{B'}(B''') \geq (1/2km)^{O(k+m)}$ such that $\|1_A \ast \beta'''\|_{L^\infty(\mu_G)} \geq \alpha(1+\Omega(\lambda))$;
\end{enumerate}
provided ${\rho'} \leq c_{\ref{cor.cs}}\lambda\alpha/d$.
\end{corollary}
\begin{proof}
We can certainly assume that all of $\lambda,\alpha$ and $\sigma$ are positive and to begin we set some parameters, the choices for which will become apparent later:
\begin{equation*}
l:=\lceil \log 2\lambda^{-1}\alpha^{-1}\rceil, p:=2+\log\alpha^{-1} \textrm{ and } \epsilon:=\lambda/8el.
\end{equation*}
The Bohr set $B'$ has dimension $O(k)$, whence we may pick $\rho''=\Omega(1/k)$ such that $B'':=B'_{\rho''/2l}$ is regular and $\mu_G(B''+B') \leq 2\mu_G(B')$.  Then
\begin{equation*}
|B''+S| \leq |B''+B'| \leq 2|B'| \leq 2\sigma^{-1}|S|,
\end{equation*}
and we apply Lemma \ref{lem.cs} to the sets $S$, $B''$ and the function $1_L$ respectively\footnote{So that $A$ is $S$, and $S$ is $B''$.} with parameters $p$ and $\epsilon$.  We get that there is an $s \in B''$ and a set $T \subset B''$ with $\beta''(T) \geq (\sigma/2)^{O(p\epsilon^{-2})}$ such that
\begin{equation*}
\|\tau_t(1_L \ast (1_S d\beta'))-  1_L\ast (1_Sd\beta')\|_{L^{p}(\mu_G)} \leq \epsilon \sigma \|1_L\|_{L^p(\mu_G)} \textrm{ for all } t \in T-s.
\end{equation*}
Of course
\begin{eqnarray*}
\|\tau_t(1_L \ast (1_S d\beta'))-  1_L\ast (1_Sd\beta')\|_{L^{p}(\beta)}^p& \leq &\frac{1}{\mu_G(B)}\int{ |\tau_t(1_L \ast (1_S d\beta'))-  1_L\ast (1_Sd\beta')|^pd\mu_G}\\ & \leq &  \epsilon^p \sigma^p \beta(L) \leq \epsilon^p \sigma^p,
\end{eqnarray*}
whence
\begin{equation*}
\|\tau_t(1_L \ast (1_S d\beta'))-  1_L\ast (1_Sd\beta')\|_{L^{p}(\beta)} \leq \epsilon \sigma \textrm{ for all } t \in T-s.
\end{equation*}
It follows by the triangle inequality that
\begin{equation*}
\|\tau_t(1_L \ast (1_S d\beta'))-  1_L\ast (1_Sd\beta')\|_{L^{p}(\beta)} \leq 2l\epsilon \sigma \textrm{ for all } t \in l(T-T).
\end{equation*}
Integrating and applying the triangle inequality again we get
\begin{equation*}
\|1_L \ast (1_S d\beta') \ast f-  1_L\ast (1_Sd\beta')\|_{L^{p}(\beta)} \leq 2l\epsilon\sigma
\end{equation*}
where $f:=\mu_T \ast \dots \ast \mu_T \ast \mu_{-T} \ast \dots \ast \mu_{-T}$ and there are $l$ copies of $\mu_T$ and $l$ copies of $\mu_{-T}$.  By H{\"o}lder's inequality we have
\begin{equation*}
|\langle 1_L \ast (1_S d\beta') \ast f ,1_{A}\rangle_{L^2(\beta)} - \langle 1_L\ast (1_Sd\beta') ,1_{A}\rangle_{L^2(\beta)}| \leq 2l\epsilon\sigma \|1_{A}\|_{L^{p/(p-1)}(\beta)} \leq \lambda\sigma \alpha/4.
\end{equation*}
It follows that we are either in the first case of the corollary or else
\begin{equation*}
\langle f \ast (1_Sd\beta') \ast 1_L ,1_{A}\rangle_{L^2(\beta)} \leq 3\lambda\sigma\alpha/4,
\end{equation*}
which we assume from hereon.

Now, $\supp f \subset 2lB'' \subset B_{\rho''}'\subset B_{\rho'}$ so
\begin{equation*}
\int{1_L \ast (1_S d\beta') \ast fd\beta} = \lambda \sigma + O({\rho'} d\sigma ),
\end{equation*}
whence
\begin{equation*}
|\langle 1_L\ast(1_Sd\beta') \ast f ,1_{A}-\alpha \rangle_{L^2(\beta)}| \geq \lambda \sigma\alpha/8
\end{equation*}
provided ${\rho'}$ is sufficiently small.  We now apply Fourier inversion to get that
\begin{equation*}
\left|\sum_{\gamma \in \wh{G}}{|\wh{\mu_T}(\gamma)|^{2l}\wh{1_Sd\beta'}(\gamma)\wh{1_L}(\gamma)\overline{(1_{A} - \alpha 1_B)^\wedge(\gamma)}}\right| \geq \lambda \sigma \alpha \mu_G(B)/8.
\end{equation*}
By the Cauchy-Schwarz inequality and the Hausdorff-Young inequality (in the trivial case which ensures $|\wh{1_Sd\beta'}(\gamma)| \leq \sigma$) we see that
\begin{equation*}
\sigma \left(\sum_{\gamma \in \wh{G}}{|\wh{1_L}(\gamma)|^2}\right)^{1/2}\left(\sum_{\gamma \in \wh{G}}{|\wh{\mu_T}(\gamma)|^{4l}|(1_{A} - \alpha 1_B)^\wedge(\gamma)|^2}\right)^{1/2} \geq \lambda\sigma \alpha \mu_G(B)/8.
\end{equation*}
Parseval's theorem tells us that the first sum is $\lambda\mu_G(B)$ and so
\begin{equation*}
\sum_{\gamma \in \wh{G}}{|\wh{\mu_T}(\gamma)|^{4l}|(1_{A} - \alpha 1_B)^\wedge(\gamma)|^2}\geq \lambda\alpha^2 \mu_G(B)/64.
\end{equation*}
We put $\eta:=(\lambda \alpha)^{1/2l}/16^{1/l}=\Omega(1)$ and since $\mu_T = 1_Td\beta''$ note, by the triangle inequality, that
\begin{equation*}
\sum_{\gamma \in  \Spec_\eta(1_T,\beta'')}{|(1_{A} - \alpha 1_B)^\wedge(\gamma)|^2}=\Omega( \lambda\alpha^2 \mu_G(B)).
\end{equation*}
The corollary is completed by Lemma \ref{lem.qw} provided ${\rho'}$ is sufficiently small.
\end{proof}

\section{Proof of the main theorem}

We shall now prove the following theorem from which our main result follows by the usual Fre{\u\i}man embedding.
\begin{theorem}\label{thm.main}
Suppose that $G$ is a group of odd order, and $A \subset G$ has density $\alpha>0$.  Then
\begin{equation*}
\langle 1_A \ast 1_{-2.A},1_{-A} \rangle_{L^2(\mu_G)} = \exp(-O(\alpha^{-1}\log^{5}2\alpha^{-1})).
\end{equation*}
\end{theorem}
There is some merit in trying to control the logarithmic term here.  Indeed, while it seems likely that with care one could improve the $5$ a bit, if one could replace it by $1-\Omega(1)$ then one could use the $W$-trick (as popularised by Green \cite{gre::8}) to deduce van der Corput's theorem pretty easily; if one could replace it by $-\Omega(1)$ then van der Corput's theorem would follow directly from the prime number theorem.

Even more ambitiously, the Erd{\H o}s-Tur{\'a}n conjecture would follow (for progressions of length three) if one could replace the $5$ by $-(1+\Omega(1))$.  However, despite the fact that such an improvement appears small it seems that a new idea would probably be required to prove such a result since it is not known even in the model setting of $G=(\Z/3\Z)^n$.  (The best result known there is the celebrated Roth-Meshulam theorem of Meshulam \cite{mes::}.)

The proof of Theorem \ref{thm.main} is an iterative application of the following lemma.
\begin{lemma}\label{lem.mainit}
Suppose that $B$ is a regular $d$-dimensional Bohr set, $B'$ is a regular rank $k$ Bohr set with $B' \subset B_{\rho'}$,$B''\subset B_{\rho''}'$, $A \subset B$ has relative density $\alpha$ and $A' \subset B'$ has relative density $\alpha'$. Then either
\begin{enumerate}
\item (large inner product)
\begin{equation*}
\langle 1_A \ast (1_{A'}d\beta'),1_{-A}\rangle_{L^2(\beta)}\geq \mu_{B'}(B'')(\alpha'/2)^{O(\alpha^{-1})};
\end{equation*}
\item (density increment) or there is a regular Bohr set $B'''$ with rank at most $k + O(\alpha^{-1}(\log^32\alpha^{-1})(\log2\alpha'^{-1}))$ and
\begin{equation*}
\mu_{B'}(B''') \geq \left(\frac{\alpha}{2k\log 2\alpha'^{-1}}\right)^{O(k+\alpha^{-1}(\log^32\alpha^{-1})(\log2\alpha'^{-1}))}
\end{equation*}
and $\|1_A \ast \beta'''\|_{L^\infty(\mu_G)} \geq \alpha(1+c_{\ref{lem.mainit}})$;
\end{enumerate}
provided $\rho' \leq c_{\ref{lem.mainit}}\alpha/d$ and $\rho'' \leq c_{\ref{lem.mainit}}\alpha'/k$.
\end{lemma}
\begin{proof}
We apply Proposition \ref{prop.inneritapplied} to see that (provided $\rho'$ and $\rho''$ aren't too large) either we are in the second case of the lemma or else there are sets $L\subset B$ and $S\subset B''$ with $\beta(L) =\Omega(1)$ and $\beta''(S) \geq (\alpha'/2)^{O(\alpha^{-1})}$ such that
\begin{equation*}
1_L \ast (1_Sd\beta'') \leq C_{\ref{prop.inneritapplied}}\alpha^{-1}\mu_{B'}(B'')^{-1} 1_A \ast (1_{A'}d\beta').
\end{equation*}
In this latter case we apply Corollary \ref{cor.cs} (to the set $-A$ provided $\rho'$ isn't too large) to get that either we are in the second case of the lemma or else
\begin{equation*}
\langle 1_L \ast (1_Sd\beta''),1_{-A} \rangle_{L^2(\beta)} \geq \alpha \beta(L)\beta''(S)/2 \geq (\alpha'/2)^{O(\alpha^{-1})},
\end{equation*}
and we are in the first case of the lemma.
\end{proof}

\begin{proof}[Proof of Theorem \ref{thm.main}]
We construct a sequence of regular Bohr sets $B^{(i)}$ and sequences
\begin{equation*}
k_i:=\rk(B^{(i)}),d_i=\dim B^{(i)} \textrm{ and } \alpha_i:=\|1_A \ast \beta^{(i)}\|_{L^\infty(\mu_G)}.
\end{equation*}
We initialise with $B^{(0)}=G$ which is easily seen to be regular so that $\alpha_0=\alpha$.  Suppose that we are at stage $i$ of the iteration.  

We have $d_i=O(k_i)$ and so by regularity we have that
\begin{equation*}
\|1_A \ast \beta^{(i)} \ast \beta^{(i)}_{\rho'} +1_{A} \ast \beta^{(i)}\ast \beta^{(i)}_{\rho'\rho''} - 2(1_A \ast \beta^{(i)})\|_{L^\infty(\mu_G)} = O(\rho'k_i).
\end{equation*}
It follows that we can pick $\rho',\rho'' = \Omega(\alpha/k_i)$ such that $B^{(i)'}:=B_{\rho'}^{(i)}$ is regular of dimension $d_i$, $B^{(i)''}:=2.B_{\rho'\rho''}^{(i)}$ is regular of rank $k_i$,
\begin{equation*}
B^{(i)''} \subset B^{(i)'}_{c_{\ref{lem.mainit}}\alpha/2d_i},
\end{equation*}
and
\begin{equation*}
\|1_A \ast \beta^{(i)} \ast \beta^{(i)'} +1_{A} \ast \beta^{(i)}\ast \beta^{(i)'}_{\rho''} - 2(1_A \ast \beta^{(i)})\|_{L^\infty(\mu_G)} \leq c_{\ref{lem.mainit}}\alpha/4.
\end{equation*}
If
\begin{equation*}
\|1_A \ast \beta^{(i)'}\|_{L^\infty(\mu_G)} \geq \alpha_i(1+c_{\ref{lem.mainit}}/4) \textrm{ or } \|1_A \ast \beta^{(i)'}_{\rho''}\|_{L^\infty(\mu_G)} \geq \alpha_i(1+c_{\ref{lem.mainit}}/4)
\end{equation*}
then we let $B^{(i+1)}$ be $B^{(i)'}$ or $B^{(i)'}_{\rho''}$ respectively and see that
\begin{equation*}
k_{i+1} =k_i, \mu_G(B^{(i+1)}) \geq \mu_G(B^{(i)})(\alpha/2k_i)^{O(k_i)} \textrm{ and } \alpha_{i+1} \geq \alpha_i(1+c_{\ref{lem.mainit}}/4).
\end{equation*}
Otherwise, by averaging, there is some $x_i$ such that
\begin{equation*}
1_A \ast \beta^{(i)'}(x_i) \geq \alpha_i(1-c_{\ref{lem.mainit}}/2) \textrm{ and } 1_A \ast \beta^{(i)'}_{\rho''}(x_i) \geq \alpha_i(1-c_{\ref{lem.mainit}}/2).
\end{equation*}
Translating by $x_i$ we get a set $A_1:=(A-x_i) \cap B^{(i)'}$ and $A_2:=(2x_i-2.A)\cap B^{(i)''}$ such that
\begin{equation*}
\beta^{(i)'}(A_1) \geq \alpha_i(1-c_{\ref{lem.mainit}}/2) \textrm{ and } \beta^{(i)''}(A_2) \geq \alpha/2,
\end{equation*}
and
\begin{equation*}
\langle 1_A \ast 1_{-2.A},1_{-A} \rangle \geq \mu_G(B^{(i)'})\mu_G(B^{(i)''})\langle 1_{A_1} \ast (1_{A_2}d\beta^{(i)''}),1_{-A_1}\rangle_{L^2(\beta^{(i)'})}.
\end{equation*}
Now we apply the preceding lemma to see that either
\begin{equation}\label{eqn.term}
\langle 1_A \ast 1_{-2.A},1_{-A} \rangle \geq \mu_G(B^{(i)'})\mu_G(B^{(i)''}_{c_{\ref{lem.mainit}}\alpha/2k_i})(\alpha/2)^{O(\alpha^{-1})},
\end{equation}
or there is a Bohr set $B^{(i+1)}$ such that
\begin{equation*}
k_{i+1} \leq k_i+O(\alpha_i^{-1}\log^42\alpha^{-1}),
\end{equation*}
\begin{equation*}
\mu_{G}(B^{(i+1)}) \geq \mu_G(B^{(i)})\left(\frac{\alpha}{2k_i}\right)^{O(k_i+\alpha_i^{-1}\log^42\alpha^{-1})},
\end{equation*}
and 
\begin{equation*}
\alpha_{i+1} \geq \alpha_i(1+c_{\ref{lem.mainit}}/2).
\end{equation*}

Since $\alpha_i$ cannot exceed $1$ the iteration described above must terminate after $i_0=O(\log2\alpha^{-1})$ steps with (\ref{eqn.term}).  By summing the geometric progression we see that
\begin{equation*}
k_{i_0} = O(\alpha^{-1}\log^42\alpha^{-1}) \textrm{ and } \mu_G(B^{(i_0)}) \geq (\alpha/2)^{O(\alpha^{-1}\log^42\alpha^{-1})}.
\end{equation*}
Inserting this in (\ref{eqn.term}) gives the required result.
\end{proof}

\section*{Acknowledgements}

The author should like to thank an anonymous referee for useful comments and suggesting the connection between the material in \S\ref{sec.kk} and the Dyson $e$-transform, Thomas Bloom for numerous corrections, Ben Green for encouragement, and Julia Wolf for many useful comments and encouragement.

\bibliographystyle{halpha}

\bibliography{references}

\end{document}